\documentclass[11pt]{amsart}

\usepackage{pgfplots}
\usepackage{enumerate}
\usepackage{url}

\usepackage[vmargin=3cm,hmargin=3cm]{geometry}

\author{Jos\'{e} I. Farr\'{a}n}
\address{Departamento de Matem\'{a}tica Aplicada, Universidad de Valladolid, Escuela de Ingenier\'{\i}a Inform\'{a}tica de Segovia, Espa\~{n}a}
\email{jifarran@eii.uva.es}
\thanks{The first author is supported by the project MTM2015-65764-C3-1-P (MINECO/FEDER)}

\author{Pedro A. Garc\'{\i}a-S\'anchez}
\address{IEMath-GR and Departamento de \'Algebra, Universidad de Granada, E-18071 Granada, Espa\~na}
\email{pedro@ugr.es} 
\thanks{The second author is supported by the projects MTM2014-55367-P, FQM-343,  FQM-5849, and FEDER funds}

\author{Benjam\'{\i}n A. Heredia}
\address{Departamento de Matem\'atica e Centro de Matem\'atica e Aplica\c{c}oes (CMA), FCT, Universidade Nova de Lisboa}
\email{b.heredia@fct.unl.pt}
\thanks{The third author is supported by the Funda\c{c}\~ao para a Ci\^encia e a Tecnologia (Portuguese Foundation for Science and Technology) through the project UID/MAT/00297/2013 (Centro de Matem\'atica e Aplica\c{c}\~oes).}

\author{Micah J. Leamer}  
\thanks{The fourth author would like to thank Marco D'Anna and the rest of orginizers of the INdAM meeting: International meeting on numerical semigroups - Cortona 2014}
\email{micahleamer@gmail.com}

\thanks{The authors would like to thank V\'\i tor Hugo Fernandes for the helpful discussions during the preparation of this paper.}

\title[The second Feng-Rao number for codes coming from inductive semigroups]{The second Feng-Rao number for codes coming from telescopic semigroups}



\theoremstyle{plain}
\newtheorem{prop}{Proposition}
\newtheorem{thm}[prop]{Theorem}
\newtheorem{lem}[prop]{Lemma}
\newtheorem{cor}[prop]{Corollary}

\theoremstyle{definition}
\newtheorem{remark}[prop]{Remark}

\newtheorem{defn}[prop]{Definition}

\newtheorem{ex}[prop]{Example}

\newcommand{\Ap}{\operatorname{Ap}}
\newcommand{\D}{\operatorname{D}}
\newcommand{\sD}{\operatorname{\#D}}
\newcommand{\sAp}{\operatorname{\#Ap}}

\newcommand{\m}{\operatorname{m}}
\newcommand{\M}{\operatorname{M}}

\begin{document}

\begin{abstract}
In this manuscript we show that the second Feng-Rao number of any telescopic numerical semigroup agrees with the multiplicity of the semigroup. To achieve this result we first study the behavior of Ap\'ery sets under gluings of numerical semigroups. These results provide a bound for the second Hamming weight of one-point Algebraic Geometry codes, which improves upon other estimates such as the Griesmer Order Bound. 
\end{abstract}

\maketitle

\section*{Introduction}

In coding theory algebraic geometry codes  (AG codes for short) are  advantageous in that their parameters asymptotically exceed the Gilbert-Varshamov bound (see~\cite{HvLP}), and due to Feng and Rao \cite{f-r} we know that they may be efficiently decoded. A bound on the number of correctable errors depends on the Feng-Rao distance of the involved Weierstrass semigroup. Even though the original codes are defined over algebraic curves, the construction can avoid the explicit use of algebraic geometry, by means of arrays of codes~\cite{K-P} and order functions over algebras~\cite{HvLP}. 

We begin with a brief overview of how Feng-Rao distances  and numerical semigroups arise in coding theory.  Let $\mathbb F_q$ be the finite field with $q$ a prime power number of elements. Let $R$ be the affine coordinate ring of a curve over $\mathbb F_q$ that is absolutely irreducible, nonsingular and with a single point at infinity. Denote the point at infinity by $Q$, and let $\mathcal P=(P_1,\ldots,P_n)$ be a list of (affine) $\mathrm F_q$-rational points on the curve. The evaluation map $\mathrm{ev}_\mathcal P : R\to \mathbb F_q^n$ is defined as $\mathrm{ev}_\mathcal P(f)=(f(P_1),\ldots, f(P_n))$. Let $\mathrm v_Q:K\to\mathbb{Z}$ be the discrete valuation at $Q$, where $K$ is the quotient ring for $R$, and define $\mathrm L(aQ)=\{f\in R\mid \mathrm v_Q(f)\ge -a\}$.  The set $\Gamma:=-\mathrm v_Q(R)$ is a numerical semigroup; see \cite{ns} for the definitions and basic properties of numerical semigroups. For $a$ in $\Gamma$, let $C_a$ be the orthogonal linear space of $\mathrm{ev}_\mathcal P(\mathrm L(aQ))$ (with respect to the usual dot product). This vector space is called the \emph{one point algebraic code} defined by $R$, $a$ and $\mathcal P$. 

The minimum distance of the code $C_a$ has a lower bound given by the \emph{Feng-Rao distance} \cite{f-r}  of $a+1$.The Feng-Rao distance is defined by $\delta_{FR}(a)=\min\{ \#\mathrm D(b)\mid a\le b, b\in \Gamma\}$, where $\mathrm D(b)=\{c\in \Gamma \mid b-c\in \Gamma\}$ denotes the set of \lq\lq divisors\rq\rq \ of $b$ in $\Gamma$ (according to the terminology in~\cite{a-g-b}).  In order to avoid ambiguity about which semigroup is being used, at times we may write $\mathrm D_{\Gamma}(b)$ instead of $\mathrm D(b)$. Let $c$ and $g$ respectively denote the \emph{conductor} and \emph{genus} of $\Gamma$.  Then for $a\ge 2c-1$ we have $\delta_{FR}(a)=a+1-2g$, which is referred to as the Goppa bound. Moreover, one has $\delta_{FR}(a)\geq a+1-2g$ for $a\geq c$. 

A natural generalization of the Feng-Rao distance is the following.  Take a sequence $a_1< \cdots < a_r$ of $r$ elements in $\Gamma$, and define $\mathrm D(a_1,\ldots, a_r)=\bigcup_{i=1}^r \mathrm D(a_i)$. The $r^{th}$ \emph{generalized Feng-Rao distance} is defined as 
\[
\delta_{FR}^r(a)=\min\{ \#\mathrm D(a_1,\ldots, a_r)\mid a_1,\ldots,a_r\in \Gamma, a\le a_1<a_2< \cdots < a_r\}.
\]  
The $r^{th}$ generalized Feng-Rao distance of $a+1$ turns out to be a lower bound for the $r^{th}$ generalized Hamming weight over the code $C_a$ (see~\cite{H-P}).
We remark that the generalized Hamming weights were introduced independently by Helleseth et al. in \cite{HKM} and 
Wei in \cite{Wei}, for applications in coding theory and cryptography respectively. 

As with the case where $r=1$, \cite[Theorem 3]{F-M} shows that for fixed $r$ the asymptotic behavior of $\delta_{FR}^r$ is linear.  In particular for all $a\ge 2c-1$ we have $\delta_{FR}^r(a)=a+1-2g+\mathrm E(\Gamma,r)$ for some constant $\mathrm E(\Gamma,r)$, known as the $r^{th}$ \emph{Feng-Rao} number of $\Gamma$. 
Furthermore, as in the classical case the inequality $\delta_{FR}^r(a)\geq a+1-2g+\mathrm E(\Gamma,r)$ holds for $a\geq c$. 
In \cite[Proposition 5]{F-M} it is shown that for $g>0$ and $r\ge 2$, we have $2\le \mathrm E(\Gamma,r)\le \rho_r$, where $\rho_r$ is the $r^{th}$ smallest element of $\Gamma$.
Clearly $\mathrm E(\Gamma,1)=0$.  

In \cite{D,F-M} it is shown that 
$\mathrm E(\Gamma,2)= \min\{\#\Ap(\Gamma,x)\mid x\in \mathbb N\setminus\{0\}\}$,
where  $\Ap(\Gamma,x)=\{a\in \Gamma\mid a-x\not\in \Gamma\}$ is the \emph{Ap\'ery set} of an integer $x$ with respect to $\Gamma$.  Many of our results are dependent on this alternative formulation.  Originally in \cite{apery} Ap\'ery only considered the case where $x$ was an element of $\Gamma$.  Since then natural generalizations have been introduced in \cite{F-M} and \cite{l-gs}, which extend the study of Ap\'ery sets allowing for $x$ to be any integer. 

Previously the cases where the second Feng-Rao number of a numerical semigroup was known were fairly limited. In \cite{F-M} it is shown that when $\Gamma$ is two generated, the second Feng-Rao number is precisely the \emph{multiplicity} of the semigroup; that is, we have $\mathrm E(\Gamma,2)=\rho_2$.  Later in \cite{D} this result was generalized to show that for every $r$ and every two generated numerical semigroup we have $\mathrm E(\Gamma,r)=\rho_r$. A formula for the $r^{th}$ Feng-Rao number was also calculated for the family of numerical semigroups generated by intervals \cite{intervalos}. Additionally an expression for $\mathrm E(\Gamma,2)$ has been found for the case where $\Gamma$ is inductive \cite{inductivos}. 

It is often the case that properties of two generated numerical semigroups generalize to the class of complete intersection numerical semigroups or more generally to the class of symmetric numerical semigroups. However, even restricting to three generated complete intersection case, there are examples where $\mathrm E(\Gamma,2)\neq\rho_2$; see for instance Example \ref{ex1}. A relevant subclass of complete intersection numerical semigroups is the class of telescopic numerical semigroups, which were introduced by Kirfel and Pellikaan in \cite{K-P} in order to study Feng-Rao distances.  This family contains the set of numerical semigroups associated to irreducible plane curve singularities, which were introduced by Zariski in \cite{zar}. Based on computational evidence it has been suggested for some time that $\mathrm E(\Gamma,2)=\rho_2$  whenever $\Gamma$ is telescopic.  In our case we used the \texttt{GAP} \cite{gap} package \texttt{numericalsgps} \cite{numericalsgps} that implements procedures from \cite{a-g} to calculate $\mathrm E(\Gamma,2)$ for all telescopic numerical semigroups with genus less than 150.  In this paper with the help of auxiliary results on Ap\'ery sets over gluings of numerical semigroups, we show that the second Feng-Rao number of any telescopic numerical semigroup agrees with its multiplicity (Corollary \ref{cor:E-tel}).  It is our hope that the results from our first section will also prove useful in deriving a formula for $\mathrm E(\Gamma,2)$ for more general cases of $\Gamma$. 

Although Feng-Rao distances and Feng-Rao numbers yield important information about AG codes, it is important to note that the numbers themselves are only dependent on the associated numerical semigroup. Consequently the focus of our research has been to develop tools that allow us to calculate properties of $\Ap(\Gamma,x)$ for a broader class of numerical semigroups.  In particular we develop formulas for calculating properties of Ap\'{e}ry sets of gluings on numerical semigroups that may be derived iteratively from the numerical semigroups in the gluing. Some of these formulas generalize a specialized result appearing in \cite[Chapter 8]{ns}.
The idea of gluing was originally developed to construct curve singularities with particular properties (see for instance \cite{B-C}) and was later generalized in \cite{delorme}, and made explicit in \cite{rosales}. 
Additionally invariants such as the Frobenius number, conductor, genus type, symmetry and Hilbert series can all be recovered from the original semigroups (see for instance \cite{gluings} and the references therein).  
 In our case by understanding Ap\'ery sets under gluings, we are able to develop tools for calculating the second Feng-Rao number for gluings of numerical semigroups in certain cases. In addition we feel that understanding the construction of Ap\'ery sets under gluings is fundamental to the theory of numerical semigroups, and will thus yield other applications in the future.

The paper is organized as follows: In section 1 we prove fundamental results concerning gluings of numerical semigroups and Ap\'{e}ry sets. Section 2 is devoted to prove the main results of the paper on the second Feng-Rao number for free numerical semigroups, and in particular Corollary \ref{cor:E-tel}. 
Section 3 computes some interesting examples of free semigroups appearing in coding theory, namely the generalized Hermitian semigroups and the Suzuki semigroups. Finally, section 4 applies the previous results to AG codes constructed from the generalized Hermitian curves and the Suzuki curves, obtaining bounds for the second Hamming weight of these codes that are better than the one given by Kirfel-Pellikaan in~\cite{K-P} and the Griesmer order bound introduced in~\cite{D}.

\section{Gluings and Ap\'ery sets}

In order to exploit the formulation of the second Feng-Rao number in terms of Ap\'ery sets, we develop some  basic properties of Ap\'ery sets and then show how they behave under gluings. We begin with some standard definitions, which mirror the notation in \cite{ns}.  Let $\mathbb{N}$ denote the non-negative integers. A numerical semigroup is a set of the form $\Gamma=\langle n_1,\hdots n_m\rangle:=n_1\mathbb{N}+\hdots+n_m\mathbb{N}$ where $n_1,\hdots, n_m$ are positive integers, such that $\gcd(n_1,\hdots,n_m)=1$. The set $A=\{n_1,\hdots,n_m\}$  is said to be the minimal generating set for $\Gamma$ provided that $\Gamma:=\langle A\rangle\neq \langle A\setminus \{n_i\}\rangle$ for any $i$.  Let $\Gamma_1$ and $\Gamma_2$ be numerical semigroups, and choose $a_1\in\Gamma_2$ and $a_2\in \Gamma_1$, such that $\gcd(a_1,a_2)=1$ and neither $a_1$ nor $a_2$ are minimal generators.  Then the set $\Gamma=a_1\Gamma_1+a_2\Gamma_2$  is again a numerical semigroup, referred to as a \emph{gluing} of $\Gamma_1$ and $\Gamma_2$.

We can say more than \cite[Proposition 1]{cmp} or \cite[Proposition 18]{F-M}.
\begin{lem}\label{apery-size}
Given a numerical semigroup $\Gamma$ and $x\in\mathbb{Z}$, we have $\#\Ap(\Gamma,x)=x+\#\Ap(\Gamma,-x)$. Thus $\#\Ap(\Gamma,x)\geq x$ with equality only if $\Ap(\Gamma,-x)=\emptyset$ or equivalently when $x\in\Gamma$. 
\end{lem}

\begin{proof}
We restrict to the case $x>0$ and the general case follows by replacing $x$ with $-x$ in the equation $\#\Ap(\Gamma,-x)=\#\Ap(\Gamma,x)-x$. 

Let $A_i$ and $B_i$ be the subsets of $\Ap(\Gamma,x)$ and $\Ap(\Gamma,-x)$ respectively whose elements are congruent to 
$i$ modulo $x$.  Notice that $A_i$ is nonemepty, since it contains the smallest element in $\Gamma$ that is congruent to $i$ modulo $x$. Let $A_i=\{a_1,a_2,\hdots,a_k\}$ with $a_1<\hdots<a_k$.   There are no elements in $B_i$ that are greater than or equal to $a_k$ because $a_k+nx-(-x)\in\Gamma$ for all $n\geq 0$.  Similarly there are no elements in $B_i$ strictly less than $a_1$.  Hence if $k=1$, we have $B_i=\emptyset$.  If $k\geq 2$, then for each $j$ with $1\leq j<k$ there exists $n_j\in\mathbb{N}$ such that $a_j+hx\in \Gamma$  for $0\leq h\leq n_j$ and $a_j+hx\notin\Gamma$ for $n_j<h<(a_{j+1}-a_j)/x$.  It follows that $b_j=a_j+n_jx$ is the unique element of $B_i$ with $a_j\leq b_j<a_{j+1}$.  Thus $|A_i|=|B_i|+1$ and $\#\Ap(\Gamma,-x)=\#\Ap(\Gamma,x)-x$.

If $x\in\Gamma$, then a standard argument shows that $\#\Ap(\Gamma,x)=x$, and hence $\#\Ap(\Gamma,-x)$.  Conversely if $x\notin\Gamma$, then $0\in\Ap(\Gamma,-x)\neq\emptyset$.
\end{proof}

\begin{lem}\label{subadditivity}
	Let $\Gamma$ be a numerical semigroup.  Then we have the following.
	\begin{enumerate}[(a)]
		\item For $x,y\in\mathbb{Z}$ we have $\sAp(\Gamma,x+y)\leq\sAp(\Gamma,x)+\sAp(\Gamma,y).$
		\item For $g, h\in\Gamma$ we have $\Ap(\Gamma,g+h)=\Ap(\Gamma,g)\cup(g+\Ap(\Gamma,h))$, and the union is disjoint.
	\end{enumerate}
\end{lem}
\begin{proof}
	Let $z\in\Ap(\Gamma,x+y)$.  Then $z\in\Gamma$.  If $z-x\notin\Gamma$, then $z\in\Ap(\Gamma,x)$.  Otherwise $z-x\in\Gamma$ and $z-x-y\notin\Gamma$; hence $z-x\in\Ap(\Gamma,y)$ and $z\in(x+\Ap(\Gamma,y))$.  Thus we have the containment $\Ap(\Gamma,x+y)\subseteq\Ap(\Gamma,x)\cup(x+\Ap(\Gamma,y))$.  This produces the inequality in (a) and one of the necessary inclusions for (b).
	
	
	It remains to show the other inclusion of sets for part (b).  Let $z'\in\Ap(\Gamma,g)$.  Then $z'\in\Gamma$, and $z'-g\notin\Gamma$.  Since $h\in\Gamma$ we have $z'-g-h\notin\Gamma$.  Thus $z'\in\Ap(\Gamma,g+h)$.  Let $z''\in g+\Ap(\Gamma,h)$.  Then $z''-g\in\Gamma$, and $z''-g-h\notin\Gamma$.  Since $g\in\Gamma$ it follows that $z''\in\Gamma$, and thus $z''\in\Ap(\Gamma,g+h)$.
\end{proof}

\begin{lem}\label{unique-form}
Let $\Gamma = a_1\Gamma_1 + a_2 \Gamma_2$ be a gluing of numerical semigroups. Any integer $z$ can be expressed uniquely as $z=a_1k+a_2\omega$ where $k\in \mathbb{Z}$ is an integer and $\omega\in \Ap(\Gamma_2,a_1)$. Then $z$ is in the semigroup $\Gamma$ if and only if $k\in \Gamma_1$.
\end{lem}

\begin{proof}
Since $\gcd(a_1,a_2)=1$, and $a_1\in\Gamma_2$, there exists a unique $\omega\in\Ap(\Gamma_2,a_1)$ such that $z\equiv a_2\omega \mod a_1$. That means that $z=a_1k+a_2\omega$ for a unique $k\in \mathbb{Z}$.

For the second part, it is clear that if $k\in \Gamma_1$, then $z$ is in $\Gamma$. Suppose that $z\in \Gamma$, so that $z=a_1x_1+a_2x_2$ with $x_1\in\Gamma_1$ and $x_2\in \Gamma_2$. Then $x_2\equiv \omega \mod a_1$ and then $x_2=\omega+na_1$ with $n>0$. This means that $z=a_1x_1+a_2x_2=a_1x_1+a_2(\omega+na_1)=a_1(x_1+na_2)+a_2\omega$, so by uniqueness, $k=x_1+na_2\in \Gamma_1$.
\end{proof}

\begin{defn}
  Let $\Gamma$ be a numerical semigroup and $g\in \Gamma$ an element. Let us write, for $0\leq i < g$, as $\omega(i)$ the unique element in $\Ap(\Gamma,g)$ such that $\omega(i)\equiv i \mod g$. Then define the cocycle $\mathrm h_{\Gamma,g}:\mathbb{Z}_g\times \mathbb{Z}_g \to \mathbb{Z}$ as
\[
\mathrm h_{\Gamma,g}(i,j)=(\omega(i) - \omega(i+j) + \omega(j))/g.
\]
\end{defn}

\begin{remark} \label{lemOmega}
From the cocycles  $\mathrm h_{\Gamma,g}(i,j)$ we can recover the elements of $\Ap(\Gamma,g)$ up to congruence. Indeed the formula \[\omega(i)=\sum\nolimits_{j=0}^{g-1}\mathrm{h}_{\Gamma,g}(j,i)\] follows easily from the definitions.
%
\end{remark}

\begin{lem} \label{lem-cocycle}
Let $\Gamma=a_1\Gamma_1 + a_2\Gamma_2$ be a gluing of numerical semigroups. Let $z = a_1k + a_2\omega(i)$ be any integer, with $\omega(i)\in \Ap(\Gamma_2,a_1)$. Then
\[
\Ap(\Gamma,z)=\bigcup\nolimits_{j=0}^{a_1-1}a_1\Ap(\Gamma_1,k+a_2\mathrm h_{\Gamma_2,a_1}(j-i,i))+a_2\omega(j),
\]
and
\[
\D_\Gamma(z)=\bigcup\nolimits_{j=0}^{a_1-1}a_1 \D_{\Gamma_1}( k-a_2\mathrm h_{\Gamma_2,a_1}(i-j,j))+a_2\omega(j).
\]
\end{lem}
\begin{proof}
Let $s=a_1p + a_2\omega(j)$ an integer. Then 
\begin{align*}
s-z&=a_1(p-k) + a_2(\omega(j)-\omega(i))\\
&=a_1(p-k)+a_2(\omega(j-i)-a_1\mathrm h_{\Gamma_2,a_1}(j-i,i))\\
&=a_1(p-k-a_2\mathrm h_{\Gamma_2,a_1}(j-i,i)) + a_2\omega(j-i).
\end{align*}
Therefore $s\in\Ap(\Gamma,z)$ if and only if $p\in\Ap(\Gamma_1,k+a_2\mathrm h_{\Gamma_2,a_1}(j-i,i))$.  Similarly
\begin{align*}
z-s&=a_1(k-p) + a_2(\omega(i)-\omega(j))\\
&=a_1(k-p)+a_2(\omega(i-j)-a_1\mathrm h_{\Gamma_2,a_1}(i-j,j))\\
&=a_1(k-a_2\mathrm h_{\Gamma_2,a_1}(i-j,j)-p) + a_2\omega(i-j),
\end{align*}
so $s\in\D_\Gamma(z)$ if and only if $p\in\D_{\Gamma_1}( k-a_2\mathrm h_{\Gamma_2,a_1}(i-j,j))$.
\end{proof}

\begin{prop}  
Let $\Gamma=a_1\Gamma_1 + a_2\Gamma_2$ be a gluing of numerical semigroups. Let $z = a_1\alpha + a_2\omega(i)$ be any integer, with $\omega(i)\in \Ap(\Gamma_2,a_1)$.
Let $\beta=\#\{x\in\Ap(\Gamma_2,a_1)|\ x-\omega(i)\notin\Gamma_2\}$.  Then
\[
\Ap(\Gamma,z)\geq (a_1-\beta)\cdot\sAp(\Gamma_1,\alpha)+\beta\cdot\sAp(\Gamma_1,\alpha+a_2).
\]  
\end{prop}

\begin{proof}
We have the following sequence of equalities:
\begin{eqnarray*}
\omega(j)-\omega(i)\in\Gamma_2 &\iff& \omega(j)-\omega(i)\in \Ap(\Gamma_2,a_1)\\
&\iff &\omega(j)-\omega(i)=\omega(j-i)\\
&\iff & \mathrm h_{\Gamma_2,a_1}(j-i,i)=(\omega(j-i)-\omega(j)+\omega(i))/a_1=0.
\end{eqnarray*}
Thus $\beta=\#\{j\in\mathbb{Z}_{a_1}|\ \mathrm h_{\Gamma_2,a_1}(j-i,i))\geq 1\}$.  Since $a_2\in\Gamma_1$ we have
 $\#\Ap(\Gamma_1,\alpha+a_2h)\geq \#\Ap(\Gamma_1,\alpha+a_2)$ when $h\geq 1$.  This explains the second step below.  
 The first step below comes directly from Lemma \ref{lem-cocycle}.
 \begin{align*}
\#\Ap(\Gamma,z)=&\sum_{j=0}^{a_1-1}\#\Ap(\Gamma_1,\alpha+a_2\mathrm h_{\Gamma_2,a_1}(j-i,i))\\
\geq&(a_1-\beta)\cdot\sAp(\Gamma_1,\alpha)+\beta\cdot\sAp(\Gamma_1,\alpha+a_2).\qedhere
 \end{align*}
\end{proof}

\begin{lem}\label{lem6}
Let $\Gamma$ be a numerical semigroup.  Choose $x\in\mathbb{Z}$ and $y\in\Gamma$. Then 
\[
\Ap(\Gamma,x+y)=\Ap(\Gamma,x)\cup(y+\Ap(\Gamma,x)) 
\cup(\Ap(\Gamma,y)\cap(x+\Ap(\Gamma,y)).
\]
\end{lem}

\begin{proof}
First we will show that each of the sets in the union on the right is contained on the left side. Suppose that $z\in\Ap(\Gamma,x)$.  Since $z-x\notin\Gamma$ we have $z-x-y\notin\Gamma$.  Thus $z\in\Ap(\Gamma,x+y)$. Suppose $z'\in(y+\Ap(\Gamma,x))$.  Then $z'-y\in\Ap(\Gamma,x)$; hence $z'-y\in\Gamma$ and $z'-y-x\notin\Gamma$.  Since $y\in\Gamma$ we have $z'=z'-y+y\in\Gamma$. Thus $z'\in\Ap(\Gamma,x+y)$. Lastly suppose that $z''\in \Ap(\Gamma,y)\cap(x+\Ap(\Gamma,y))$, Then $z''\in\Ap(\Gamma,y)$ implies that $z''\in\Gamma$. Also $z''-x\in\Ap(\Gamma,y)$ implies that $z''-x-y\notin\Gamma$.  Thus $z''\in\Ap(\Gamma,x+y)$. 

 Let $w\in\Ap(\Gamma,x+y)$. Assuming $w\notin\Ap(\Gamma,x)\cup(y+\Ap(\Gamma,x))$ it suffices to show that $w\in(\Ap(\Gamma,y)\cap(x+\Ap(\Gamma,y))$.  Since $w\in\Gamma$, and $w\notin\Ap(\Gamma,x)$ we have that $w-x\in\Gamma$. Thus $w-x\in\Ap(\Gamma,y)$ and $w\in x+\Ap(\Gamma,y)$. Since $w-y\notin \Ap(\Gamma,x)$ and $w-x-y\notin\Gamma$ we have $w-y\notin\Gamma$.  Thus $w\in\Ap(\Gamma,y)$ and the result follows.
\end{proof}

\begin{remark}\label{rmk7}
In Lemma \ref{lem6} we get two special cases:
\begin{enumerate}
\item \label{rmk7-1}If $-x\in\Gamma$, then
$\Ap(\Gamma,x+y)=\Ap(\Gamma,y)\cap(x+\Ap(\Gamma,y))$;
\item \label{rmk7-2} For a set of integers $X$, write $\Delta(X)=\{x'-x \mid x,x'\in X, x'>x \}$.\\
 If $|x|\notin\Delta(\Ap(\Gamma,y))\cup\{0\}$,  then 
$\Ap(\Gamma,x+y)=\{0,y\}+\Ap(\Gamma,x)$.
\end{enumerate}
\end{remark}

For a gluing $\Gamma=a_1\Gamma_1+a_2\Gamma_2$ the next result produces a simple formula for $\Ap(\Gamma,z)$ in the case where $z$ is divisible by $a_1$. 

\begin{prop}\label{apery-multiple}
Let $\Gamma=a_1\Gamma_1+a_2\Gamma_2$ be a gluing of numerical semigroups. Then given any $z\in\mathbb{Z}$ we have
\[
\Ap(\Gamma,a_1z)=a_1\Ap(\Gamma_1,z)+a_2\Ap(\Gamma_2,a_1).
\] 
In particular
\[
\sAp(\Gamma,a_1z)= a_1\cdot\sAp(\Gamma_1,z).
 \] 
\end{prop}

%

\begin{proof}
Using the unique form in Lemma \ref{unique-form} we have $a_1z=a_1k+a_2\omega$ with $\omega=\omega(0)=0$ and $z=k$.  Notice that $\mathrm h_{\Gamma_2,a_1}(j-0,0)=0$ for all $j$; hence by Lemma \ref{lem-cocycle} we have
\[
\Ap(\Gamma,a_1z)=\bigcup_{j=0}^{a_1-1}a_1\Ap(\Gamma_1,z)+a_2\omega(j)=a_1\Ap(\Gamma_1,z)+a_2\Ap(\Gamma_2,a_1).\qedhere
\]
\end{proof}

Notice that by setting $z=a_2$ in Proposition \ref{apery-multiple} we can recover the previously known special case  $\Ap(\Gamma,a_1a_2)=a_1\Ap(\Gamma_1,a_2)+a_2\Ap(\Gamma_2,a_1)$ from \cite[Chapter 8]{ns}.

Combining Lemma \ref{subadditivity}~(c) with Propostion \ref{apery-multiple} we get an expression for $\Ap(\Gamma,z)$ when $z\in\Gamma$. 

\begin{cor} \label{cor10}
Let $\Gamma=a_1\Gamma_1+a_2\Gamma_2$ be a gluing of numerical semigroups. Let $g_1\in\Gamma_1$, $g_2\in\Gamma_2$.  and $g=a_1g_1+a_2g_2\in\Gamma$.   Then
\[
\Ap(\Gamma,g)=(a_1\Ap(\Gamma_1,g_1)+a_2\Ap(\Gamma_2,a_1))
\cup(a_1g_1+a_1\Ap(\Gamma_1,a_2)+a_2\Ap(\Gamma_2,g_2)).
\]
\end{cor}

Since $\sAp(\Gamma,g)=g$ for any $g\in\Gamma$, it follows that the union in Corollary \ref{cor10} is disjoint.

\begin{thm}\label{lem11}
Let $\Gamma=a_1\Gamma_1+a_2\mathbb{N}$ be a gluing of numerical semigroups, and let $z$ be an integer. Express $z$ as $z=a_1\alpha+a_2\beta$ with $0\leq \beta<a_1$ and $\alpha\in\mathbb{Z}$.   We have the following: 
\begin{enumerate}
\item $\Ap(\Gamma,z)=(a_1\Ap(\Gamma_1,\alpha+a_2)+a_2\{0,\hdots,\beta-1\})\cup(a_1\Ap(\Gamma_1,\alpha)+a_2\{\beta,\hdots, 
a_1-1\});$\\

\item $\D_\Gamma(z)=(a_1\D_{\Gamma_1}(\alpha)+a_2\{0,\hdots,\beta\})\cup (a_1\D_{\Gamma_1}(\alpha+a_2)+a_2\{\beta+1,\hdots,
a_1-1\});$\\

\item $\sAp(\Gamma,z)=\beta\cdot\sAp(\Gamma_1,\alpha+a_2)+(a_1-\beta)\cdot\sAp(\Gamma_1,\alpha);$\\

\item $\sD_\Gamma(z)=(\beta+1)\sD_{\Gamma_1}(\alpha)+(a_1-\beta-1)\sD_{\Gamma_1}(\alpha-a_2)$. 
\end{enumerate}
\end{thm}

%
%
\begin{proof}
Notice that $\Ap(\mathbb{N},a_1)=\{0,1,\hdots,a_1-1\}$ and $\omega(j)=j$ for $0\leq j<a_1$.  When $j<\beta$ we have $\mathrm h_{\mathbb{N},a_1}(j-\beta,\beta)=1$, and when $j\geq \beta$ we have 
$\mathrm h_{\mathbb{N},a_1}(j-\beta,\beta)=0$. By Lemma \ref{lem-cocycle} we have
\begin{align*}
\Ap(\Gamma,z)&=\bigcup\nolimits_{j=0}^{a_1-1}a_1\Ap(\Gamma_1,\alpha+a_2\mathrm h_{\Gamma_2,a_1}(j-\beta,\beta)))+a_2j\\
&= (a_1\Ap(\Gamma_1,\alpha+a_2)+a_2\{0,\hdots,\beta-1\})\cup (a_1\Ap(\Gamma_1,\alpha)+a_2\{\beta,\hdots, 
a_1-1\}).
\end{align*}
When $j\leq \beta$ we have $\mathrm h_{\mathbb{N},a_1}(\beta-j,j)=0$, and when $j>\beta$ we have 
$\mathrm h_{\mathbb{N},a_1}(\beta-j,j)=1$. By Lemma \ref{lem-cocycle} we have
\begin{align*}
\D_\Gamma(z)&=\bigcup\nolimits_{j=0}^{a_1-1}a_1\D_{\Gamma_1}(\alpha-a_2\mathrm h_{\Gamma_2,a_1}(\beta-j,j)))+a_2j\\
&=(a_1\D_{\Gamma_1}(\alpha)+a_2\{0,\hdots,\beta\})\cup (a_1\D_{\Gamma_1}(\alpha-a_2)+a_2\{\beta+1,\hdots, 
a_1-1\}).
\end{align*}

Hence  $\sD_{\Gamma}(z)=(\beta+1)\sD_{\Gamma_1}(\alpha)+(a_1-\beta-1)\sD_{\Gamma_1}(\alpha-a_2)$. 
\end{proof}

Notices that (4) in Theorem \ref{lem11} can be sharpened a bit more. If $\alpha\in \Ap(\Gamma_1,a_2)$, then $\D_{\Gamma_1}(\alpha-a_2)$ is empty. If to the contrary, $\alpha-a_2\in \Gamma$, then $\D_{\Gamma_1}(\alpha)=\D_{\Gamma_1}(\alpha-a_2)\cup\D_{\Gamma_1}(\alpha)\setminus \D_{\Gamma_1}(\alpha-a_2)$. And the cardinality of $\D_{\Gamma_1}(\alpha-a_2+a_2)\setminus \D_{\Gamma_1}(\alpha-a_2)$ is precisely that of $\Ap(\Gamma,a_2)$ (\cite[Prposition 11]{D}). Hence $\sD_\Gamma(z)=(\beta+1)\sD_{\Gamma_1}(\alpha)+(a_1-\beta-1)\sD_{\Gamma_1}(\alpha-a_2)=(\beta+1)(\sD_{\Gamma_1}(\alpha-a_2)+a_2)+(a_1-\beta-1)\sD_{\Gamma_1}(\alpha-a_2)=a_1\sD_{\Gamma_1}(\alpha-a_2)+(\beta+1)a_2$. By putting all this together, we obtain.
\begin{equation}
\sD(z)=\begin{cases}
(\beta+1)\sD_{\Gamma_1}(\alpha), \hbox{ if }\alpha\in \Ap(\Gamma,a_2),\\
a_1\sD_{\Gamma_1}(\alpha-a_2)+(\beta+1)a_2, \hbox{ otherwise}.
\end{cases}
\end{equation}

\begin{remark}\label{remark12}
Suppose $\Gamma=\langle a,b\rangle$ and $z=ua+vb$ with $0\leq u< b$. Then we may apply Theorem \ref{lem11} with
$\Gamma_1=\mathbb{N}$, $a_1=b$, $a_2=a$, $\beta=u$ and $\alpha=v$.
Thus 
\[
\sAp(\Gamma,z)=(b-u)\max\{v,0\}+u\max\{v+a,0\}.
\] 
This yields the same result as \cite[Theorem 14]{D}.
\end{remark}

\section{The second Feng-Rao number for free numerical semigroups}


Recall that the second Feng-Rao number of a numerical semigroup $\Gamma$ can be computed as 
\begin{align*}
\mathrm E(\Gamma,2)=&\min\{\#\mathrm{Ap}(\Gamma, z)\mid  z\in \mathbb N\setminus\{0\}\}\\
=&\min\{ \#\Ap(\Gamma,z)\mid z\in \{1,\ldots,\mathrm m(\Gamma)\}\}.
\end{align*}
To realize the second equality above notice that for $x\ge \mathrm m(\Gamma)$ we have $\#\Ap(\Gamma,x)\ge x\ge \mathrm m(\Gamma)$.
\begin{remark}
Let $\Gamma$ be a numerical semigroup.  Then $\mathrm E(\Gamma,2)=\mathrm m(\Gamma)$ if and only if there exists $y\geq\mathrm m(\Gamma)$ such that $\sAp(\Gamma,z)\geq \mathrm m(\Gamma)\lceil \frac{z}{y}\rceil$ for all $z\in\mathbb{Z}$.  If this is the case, then $y$ can be chosen so that it is at most $2\mathrm m(\Gamma)-1$. 

Clearly if $\sAp(\Gamma,z)\geq \mathrm m(\Gamma)\lceil \frac{z}{y}\rceil$ with $y> 0$, then $\sAp(\Gamma,z)\geq \mathrm m(\Gamma)$ for $z>0$; hence $\mathrm E(\Gamma,2)=\mathrm m(\Gamma)$.  Let $z=n\mathrm m(\Gamma)+r$ with $0\leq r<\mathrm m(\Gamma)$.  Notice that for $n\geq 1$ we always have $n\geq \lceil\frac{n\mathrm m(\Gamma)+r}{2\mathrm m(\Gamma)-1}\rceil$; hence $\sAp(\Gamma,z)\geq n\mathrm m(\Gamma)+r\geq\mathrm m(\Gamma)\lceil \frac{n\mathrm m(\Gamma)+r}{2\mathrm m(\Gamma)-1}\rceil$.  Therefore supposing $\mathrm E(\Gamma,2)=\mathrm m(\Gamma)$ we have $\sAp(\Gamma,z)\geq \mathrm m(\Gamma)\lceil \frac{z}{2\mathrm m(\Gamma)-1}\rceil$ for all $z\in\mathbb{Z}$.  Also notice that if the statement is true for a certain $y$ value then it is true for all $y'>y$.  If $0<y<\mathrm m(\Gamma)$, then  $\sAp(\Gamma,\mathrm m(\Gamma))=\mathrm m(\Gamma)<\mathrm m(\Gamma)\lceil \frac{\mathrm m(\Gamma)}{y}\rceil$  would not work; hence we must have $y\geq\mathrm m(\Gamma)$.
\end{remark}

\begin{thm}\label{thm16}
Let $\Gamma=a_1\Gamma_1+a_2\Gamma_2$ be a gluing of numerical semigroups. Suppose there exists $y\geq \mathrm m(\Gamma)$ such that $\sAp(\Gamma_1,z)\geq \mathrm m(\Gamma_1)\lceil\frac{z}{y}\rceil$ for all $z\in\mathbb{Z}$.  If $a_2>a_1y$, then 
$\sAp(\Gamma,z)\geq \mathrm m(\Gamma)\lceil \frac{z}{a_2}\rceil$ for all $z\in\mathbb{Z}$
\end{thm}

\begin{proof}
If $z$ is less than or equal to zero the result is clear.  Let $z$ be a positive integer and write $z=a_1k+a_2\omega(i)$, where $\omega(j)$ denotes the element of $\Ap(\Gamma_2,a_1)$ which is congruent to $j$ modulo $a_1$.  For ease of notation set $h_j=\mathrm h_{\Gamma_2,a_1}(j-i,i)$.  As in Lemma \ref{lemOmega} note that $\sum_{j=0}^{a_1-1}h_j=\omega(i)$.
First suppose $k\geq 0$.  Lemma \ref{lem-cocycle} gives us the first step below.
\begin{eqnarray*}
\sAp(\Gamma,z)&=&\textstyle\sum_{j=0}^{a_1-1}\sAp(\Gamma_1,k+a_2h_j)\\
&\geq& \textstyle\sum_{j=0}^{a_1-1}\mathrm m(\Gamma_1)\lceil \frac{k+a_2h_j}{y}\rceil \\
&\geq&\textstyle\sum_{j=0}^{a_1-1}\mathrm m(\Gamma_1)\lceil \frac{k+a_1yh_j}{y}\rceil )\\
&=&\textstyle\sum_{j=0}^{a_1-1}\mathrm m(\Gamma_1)(\lceil \frac{k}{y}\rceil + a_1h_j) \\
&=&\textstyle a_1\mathrm m(\Gamma_1)\lceil \frac{k}{y}\rceil + a_1\mathrm m(\Gamma_1)\omega(i)\\
&=&\textstyle a_1\mathrm m(\Gamma_1)\lceil \frac{k}{y}+\omega(i)\rceil\\
&=&\textstyle \mathrm m(\Gamma)\lceil \frac{a_1k}{a_1y}+\frac{a_2\omega(i)}{a_2}\rceil\\
&\geq&\textstyle \mathrm m(\Gamma)\lceil \frac{a_1k}{a_2}+\frac{a_2\omega(i)}{a_2}\rceil\\
&\geq&\textstyle \mathrm m(\Gamma)\lceil \frac{z}{a_2}\rceil.
\end{eqnarray*}
Next supoose $k<0$.  Set $c=a_2-a_1y$.   From the equation $z=a_1k+a_2\omega(i)$ we get 
$\lceil\frac{z}{a_2}\rceil=\omega(i)-\lfloor\frac{-a_1k}{a_2}\rfloor$, which gives us the last step below.  Since $0<c<a_2$ when $h_j\geq\frac{-k}{a_2}$ we have $\frac{-k}{a_2}\geq\frac{-k-h_jc}{a_2-c}$, which explains the eighth step below.  
\begin{align*}
\sAp(\Gamma,z)&=\textstyle\sum_{j=0}^{a_1-1}\sAp(\Gamma_1,k+a_2h_j)\\
&\geq\textstyle\sum_{h_j\geq\frac{-k}{a_2}}\sAp(\Gamma_1,k+a_2h_j)\\
&\geq\textstyle\mathrm m(\Gamma_1)\sum_{h_j\geq\frac{-k}{a_2}}\left\lceil \frac{k+a_2h_j}{y}\right\rceil \\
&=\textstyle\mathrm m(\Gamma_1)\sum_{h_j\geq\frac{-k}{a_2}}\left\lceil \frac{k+h_j(c+a_1y)}{y}\right\rceil \\
&=\textstyle\mathrm m(\Gamma_1)\sum_{h_j\geq\frac{-k}{a_2}}\left(a_1h_j+\left\lceil \frac{k+h_jc}{y}\right\rceil\right) \\
&=\textstyle\mathrm m(\Gamma_1)\sum_{h_j\geq\frac{-k}{a_2}}\left(a_1h_j-\left\lfloor \frac{-k-h_jc}{y}\right\rfloor\right) \\
&=\textstyle\mathrm m(\Gamma_1)\sum_{h_j\geq\frac{-k}{a_2}}\left(a_1h_j-\left\lfloor a_1\frac{-k-h_jc}{a_2-c}\right\rfloor\right) \\
&\geq \textstyle\mathrm m(\Gamma_1)\sum_{h_j\geq\frac{-k}{a_2}}\left(a_1h_j-\left\lfloor a_1\frac{-k}{a_2}\right\rfloor\right) \\
&= \textstyle\mathrm m(\Gamma_1)\left(\sum_{j=0}^{a_1-1}a_1h_j-\sum_{h_j<\frac{-k}{a_2}}a_1h_j-\sum_{h_j\geq\frac{-k}{a_2}}\left\lfloor\frac{-a_1k}{a_2}\right\rfloor\right) \\
&\geq \textstyle\mathrm m(\Gamma_1)\left(a_1\omega(i)-\sum_{h_j<\frac{-k}{a_2}}\left\lfloor\frac{-a_1k}{a_2}\right\rfloor-\sum_{h_j\geq\frac{-k}{a_2}}\left\lfloor\frac{-a_1k}{a_2}\right\rfloor\right) \\
&= \textstyle\mathrm m(\Gamma_1)\left(a_1\omega(i)-a_1\left\lfloor\frac{-a_1k}{a_2}\right\rfloor\right) \\
&= \textstyle \mathrm m(\Gamma)\lceil \frac{z}{a_2}\rceil. \qedhere
\end{align*}
\end{proof}

We say that $\Gamma$ is \emph{free} if either $\Gamma$ is $\mathbb N$ or it is the a gluing of a free numerical semigroup with $\mathbb N$ (these semigroups were introduced by \cite{B-C}).  Analogously, a numerical semigroup $\Gamma$ is \emph{telescopic} if either $\Gamma$ is $\mathbb N$ or $\Gamma=a_1\Gamma_1+a_2\mathbb N$ is a gluing with $\Gamma_1$ telescopic and $a_2>a_1 n$ for each minimal generator $n$ of $\Gamma_1$. 

\begin{cor}\label{cor:E-tel}
Let $\Gamma=\langle n_1,n_2,\dots,n_m\rangle$ be telescopic with $n_1<\cdots<n_m$.  Given $z\in\mathbb{N}$ we have $\sAp(\Gamma,z)\geq n_1\lceil\frac{z}{n_m}\rceil$; hence $\mathrm E(\Gamma,2)=n_1$.
\end{cor}

\begin{proof}
We proceed by induction on $m$. For $\Gamma=\mathbb{N}$, the result is obviously true. Now suppose $\Gamma=\langle n_1,n_2, \dots, n_m\rangle$ is telescopic, that is $\Gamma=d\Gamma_1 + n_m\mathbb{N}$ where $d=\gcd(n_1,\dots,n_{m-1})$, $\Gamma_1=\langle n_1/d,\dots,n_m/d\rangle$ is telescopic, and $n_m> n_{m-1}$. By induction hypothesis $\#\Ap(\Gamma_1,z)\geq m(\Gamma_1)\lceil\frac{z}{n_{m-1}/d}\rceil$ for all $z\in \mathbb{Z}$, and so by Theorem \ref{thm16} we get $\#\Ap(\Gamma,z)\geq m(\Gamma)\lceil \frac{z}{n_m}\rceil$.
\end{proof}

If we take $y=1$ and $\Gamma_1=\mathbb N$ in Theorem \ref{thm16}, we obtain the following consequence.
\begin{cor}
Let $\Gamma=a_1\mathbb{N}+a_2\Gamma_2$ be a gluing of numerical semigroups.  If $a_2>a_1$, then 
$\sAp(\Gamma,z)\geq  a_1\lceil \frac{z}{a_2}\rceil$ for all $z\in\mathbb{Z}$; hence $\mathrm E(\Gamma,2)=a_1=\mathrm m(\Gamma)$. 
\end{cor}

As we mentioned above, if $\Gamma$ has embedding dimension two, then $\mathrm E(\Gamma,r)$ corresponds with the $r^{th}$ smallest entry of $\Gamma$. However this is no longer true for $r>2$ in a telescopic numerical semigroup.

\begin{ex}
Let $\Gamma=\langle  6, 10, 11 \rangle$. Then $\Gamma$ is a telescopic numerical semigroup, 
\[\Gamma =\{0, 6, 10, 11, 12, 16, 17, 18, 20, 21, 22, 23, 24, 26,\to\},\] and $\mathrm E(\Gamma,3)=9<10$.
\end{ex}

By \cite[Chapter 8]{ns} a numerical semigroup $\Gamma\neq \mathbb{N}$ is free if and only if there exists an arrangement of its minimal generators $(n_1,\dots, n_e)$ such that for
$d_i:=\gcd(n_1,\ldots, n_{i-1})$ and $c_i:=d_i/d_{i+1}$
we have $d_i>d_{i+1}$ and $c_in_i\in \langle n_1,\ldots, n_{i-1}\rangle$ for $i\in\{2,\ldots, e\}$.
In this case we will say that $\Gamma$ is free with respect to the arrangement $(n_1,\ldots, n_e)$. 

\begin{lem}\label{lemBound}
Let $\Gamma=a_1\Gamma_1+a_2\mathbb{N}$ be a gluing of numerical semigroups and $z$ a positive integer.  We may write $z=a_1\alpha+a_2\beta$ with $0\leq \beta<a_1$ and $\alpha\in\mathbb{Z}$.  Then we have the following:
\[
\textstyle\sAp(\Gamma,z)>(a_1-\beta)\left(\sAp(\Gamma_1,\alpha)+\frac{a_2\beta}{a_1}\right).
\]
\begin{enumerate}[$\bullet$]
\item If $\alpha> 0$ (this includes the case $\beta=0$), then $\sAp(\Gamma,z)\geq a_1\mathrm{E}(\Gamma_1,2)$.
\item If $\beta\in\{2,3,\hdots, a_1-2\}$, then $\sAp(\Gamma,z)> a_2$.
\item If $\beta\in\{1,a_1-1\}$, then $\sAp(\Gamma,z)> \frac{a_1-1}{a_1}a_2$.
\end{enumerate}
 \[\textrm{Thus }\quad\textstyle\mathrm{E}(\Gamma,2)\geq\min\left\{a_1\mathrm E(\Gamma_1,2),\ \frac{(a_1-1)a_2}{a_1}\right\}.\]
\end{lem}

\begin{proof}
By Theorem \ref{lem11}, we have the first step below.
\begin{align*}
\sAp(\Gamma,z)&=(a_1-\beta)\sAp(\Gamma_1,\alpha)+\beta\cdot\sAp(\Gamma_1,\alpha+a_2)\\
&\geq (a_1-\beta)\sAp(\Gamma_1,\alpha)+\beta(\alpha+a_2)\\
& > \textstyle (a_1-\beta)\sAp(\Gamma_1,\alpha)+\beta\left(-\frac{a_2\beta}{a_1}+a_2\right)\\
& = \textstyle (a_1-\beta)\left(\sAp(\Gamma_1,\alpha)+\frac{a_2\beta}{a_1}\right).
\end{align*}
The second step above follows from Lemma \ref{apery-size}.  The third step follows from the assumption $z=a_1\alpha+a_2\beta>0$, and the last step is elementary algebra. 

If $\alpha>0$, then $\sAp(\Gamma_1,\alpha+a_2)\geq\sAp(\Gamma_1,\alpha)\geq\mathrm{E}(\Gamma_1,2)$; hence 
\[\sAp(\Gamma,z)\geq (a_1-\beta)\mathrm{E}(\Gamma_1,2)+\beta\mathrm{E}(\Gamma_1,2)=a_1\mathrm{E}(\Gamma_1,2).\]
If $\beta\in\{2,3,\hdots, a_1-2\}$, then $a_1\geq 4$.  Since $\beta(a_1-\beta)$ has a unique local maximum at $\beta=\frac{a_1}{2}$ we have $\beta(a_1-\beta)\geq2(a_1-2)$.  Since $a_1\geq 4$, we have $2(a_1-2)\geq a_1$.  Thus 
\[ 
\textstyle\sAp(\Gamma,z)> \beta(a_1-\beta)\frac{a_2}{a_1}\geq a_1\frac{a_2}{a_1}=a_2.
\]

If $\beta\in\{1,a_1-1\}$, then $(a_1-\beta)\beta=a_1-1$. Thus $\sAp(\Gamma,z)>(a_1-\beta)\frac{a_2\beta}{a_1}=\frac{a_1-1}{a_1}a_2$.
\end{proof}

\begin{thm}\label{ap-min}
Let $\Gamma$ be a free numerical semigroup with respect to the arrangement 
$(n_1,\ldots, n_e)$ of its minimal generators. If $\frac{c_i-1}{c_i} n_i\ge 
n_1$ for $i\in\{3,\ldots,e\}$, then 
$\mathrm E(\Gamma,2)=\m(\Gamma)$.
\end{thm}
\begin{proof}
For $e=1$ we have $\Gamma=\mathbb N$. For $e=2$ the result is known to be true; see for instance \cite{D}.  Suppose that $e\geq 3$ and that the result is true for the case with $e-1$ generators.  Then $\mathrm E(\Gamma',2)=\m(\Gamma')=\frac{n_1}{d_e}$ since the sequence of $c_i$'s for $\Gamma'$ is that of $\Gamma$ after removing  the last one.  Note that $c_e=d_e$.  We  have that $\Gamma=c_e\Gamma'+n_e\mathbb{N}$.   We have
$\mathrm E(\Gamma,2)\geq\min\{c_e\mathrm E(\Gamma',2),\ \frac{(c_e-1)n_e}{c_e}\}$ by Lemma \ref{lemBound}.  Note that $c_e\mathrm E(\Gamma',2)=n_1=\m(\Gamma)$ and by hypothesis $\frac{(c_e-1)n_e}{c_e}\geq n_1=\m(\Gamma)$.  Thus $\mathrm E(\Gamma,2)\geq \m(\Gamma)$.  Since we always have $\mathrm E(\Gamma,2)\leq\m(\Gamma)$, the result follows.
\end{proof}


If we remove the condition $\frac{c_i-1}{c_i}n_i\ge n_1$, it may happen that $\mathrm E(\Gamma,2)<\mathrm m(\Gamma)$. 

\begin{ex} \label{ex1} Let $\Gamma=\langle 4,5,6\rangle$. Then $\Gamma=2\langle 2,3\rangle+5\mathbb N$ is a gluing of numerical semigroups. Notice that $\Gamma$ is free but not telescopic. Also  $\frac{1}{2}5=\frac{c_3-1}{c_3}n_3< n_1=4$. In this case we have 
\[\mathrm E(\Gamma,2)=\sAp(\Gamma,1)=|\{0, 4, 8\}|=3<4=\m(\Gamma).\]
\end{ex}

\begin{remark}
It is possible to apply Theorem \ref{lem11}  iteratively to find explicit formulas for $\sAp(\Gamma,z)$ when $\Gamma$ is a free numerical semigroups with more than two generators.  However, the complexity of these formulas tends to increase exponentially with the number of generators. Let $\M(x)=\max\{x,0\}$
 
For example suppose $\Gamma=\langle a,b,c\rangle$ is a three generated complete intersection numerical semigroup (and thus free).  Then up to a permutation of the generators  we may write $a=\sigma x$, $b=\sigma y$ and $c=c_x x+c_y y$ where $\sigma,x,y>1$, $x>c_y\geq 0$, $c_x\geq 0$, $\gcd(x,y)=1$ and $\gcd(\sigma,c)=1$. Given any integer $z$ we may write $z$ uniquely as 
	$z=z_\sigma\sigma+z_cc$ with $0\leq z_c< \sigma$, $z_\sigma=z_xx+z_yy$ and $0\leq z_y< x$. 
	
	If $z_y+c_y<x$, then $\sAp(\Gamma,z)$ may be expressed as
	\begin{multline*}
	(\sigma-z_c)\big((x-z_y)\M(z_x)+z_y\M(z_x+y)\big) \\
	 +\ 
	z_c\big((x-z_y-c_y)\M(z_x+c_x)+(u_y+c_y)\M(z_x+c_x+y)\big).
	\end{multline*}
	
	If $z_y+c_y\geq x$, then $\sAp(\Gamma,z)$ may be expressed as
	\begin{multline*}
	(\sigma-z_c)\left((x-z_y)\M(z_x)+z_y\M(z_x+y)\right) \\
	\quad +\ z_c\left((2x-z_y-c_y)\M(2z_x+y)+(z_y+c_y-x)\M(z_x+c_x+2y)\right).
	\end{multline*}	
	
\end{remark}

\section{Some interesting examples}

In this section, we apply the above calculations to some interesting examples, coming from the theory of AG codes. 

\subsection{Generalized Hermitian semigroups}\label{sec:GH}

As a generalization of classical Hermitian semigroups, $\langle q,q+1\rangle$ with $q$ an integer greater than 2, we consider generalized Hermitian semigroups. These are three generated semigroups depending on two parameters: $q$ as above and an integer $r>2$. The generalized Hermitian curve $\chi_{r}$ with parameters $q$ and $r$, is defined over $\mathbb{F}_{q^{r}}$, by the equation 
\[
Y^{q^{r-1}}+\cdots+Y^{q}+Y=X^{1+q}+\cdots+X^{q^{r-2}+q^{r-1}},
\]
and has $q^{2r-1}+1$ rational points over $\mathbb{F}_{q^{r}}$. Its Weierstrass semigroup at the unique pole of $X$ is precisely $H_{q,r}=\langle q^{r-1}, q^{r-1}+q^{r-2}, q^r+1 \rangle$.  For background on this topic see \cite{castle}. Note that $\chi_{2}$ is the classical Hermitian curve. 

Observe that $H_{q,r}$ is the gluing of $\langle q,q+1\rangle$ and $\mathbb N$ given by $H_{q,r}=q^{r-2}\langle q,q+1\rangle+(q^r+1)\mathbb N$.  Thus these are telescopic numerical semigroups, and Corollary \ref{cor:E-tel} yields 
$\mathrm E(H_{q,r},2)=q^{r-1}$. 

\subsection{Generalized Suzuki numerical semigroups}\label{sec:Suzuki}

Given positive integers $p$ and $n$, define the generalized Suzuki numerical semigroup as 
\[
S_{p,n} = \langle p^{2n+1}, p^{2n+1}+p^n, p^{2n+1}+p^{n+1},p^{2n+1}+p^{n+1}+1\rangle
\]
(see \cite{suzuki, gen-suzuki}). The cases where $p=2$ and $n$ varies are called Suzuki numerical semigroups.  Suzuki numerical semigroups come from the Suzuki curve $\chi^{n}$ defined by the equation 
\[
Y^{2^{2n+1}}+Y=X^{2^{2n}}(X^{2^{2n+1}}+X).
\] 
These curves have numerous rational points over $\mathbb{F}_{2^{2n+1}}$, which makes them useful for coding theory purposes. 
By defining 
$\Gamma_1(p,n):= \langle p^{n+1}, p^{n+1}+1, p^{n+1}+p\rangle$, we obtain the gluing 
\[
S_{p,n} = p^n \Gamma_1(p,n) + (p^{2n+1} + p^{n+1}+1)\mathbb{N}.
\]
Notice that $\Gamma_1(p,n)= p \langle p^n, p^n+1\rangle + (p^{n+1}+1)\mathbb{N}$ is also a gluing
but is not telescopic.

\begin{lem}
$\mathrm E(\Gamma_1(p,n),2) = p^{n+1} -p^n +1$.
\end{lem}
\begin{proof}
Notice that 
\[
p \mathrm{E}(\langle p^n, p^n+1\rangle,2)=p^{n+1}>p^{n+1}-p^{n}+1-\frac{1}{p}=\frac{(p-1)(p^{n+1}+1)}{p}.
\]
Thus by Lemma \ref{lemBound} we have
$\mathrm E(\Gamma_1(p,n),2)\geq p^{n+1}-p^{n}+1-\frac{1}{p}$.

By rounding up to the nearest integer we obtain $\mathrm{E}(\Gamma_1(p,n),2) \geq p^{n+1}-p^{n}+1$.

Now write $1=p(-p^{n})+ (p^{n+1}+1)$. Using the notation in Theorem \ref{lem11} for $z=1$ we have $\alpha=-p^n$, $\beta=1$, $a_1=p$ and $a_2=p^{n+1}+1$.  Thus
\begin{align*}
\#\Ap(\Gamma_1(p,n),1)&=\#\Ap(\langle p^n,p^n+1\rangle,-p^n + p^{n+1}+1) + (p-1) \#\Ap(\langle p^n,p^n+1\rangle, -p^n)\\
&=p^{n+1}-p^n+1,
\end{align*}
and the result follows.
\end{proof}

\begin{thm}\label{thm-Suzuki}
$\mathrm E(S_{p,n},2)=p^{2n+1}-p^{2n}+p^n.$
\end{thm}
\begin{proof}
Again we can compute
\[
p^n \mathrm{E}(\Gamma_1(p,n),2)=p^{2n+1}-p^{2n}+p^n.
\]
and
\[
\frac{(p^n-1)(p^{2n+1}+p^{n+1}+1)}{p^n}=
p^{2n+1}-p +1 -\frac{1}{p^n}.
\]

Since the first of these values is the smallest, Lemma \ref{lemBound} gives us the inequality
\[
\mathrm E(S_{p,n},2)\geq p^{2n+1}-p^{2n}+p^n.
\]
From Proposition \ref{apery-multiple} we obtain
$\#\Ap(S_{p,n},p^n)=p^n\#\Ap(\Gamma_1(p,n),1)=p^{2n+1}-p^{2n}+p^n$, and the result follows. 
\end{proof}

\section{Application to AG codes}\label{sec:experiments}

Corollary \ref{cor:E-tel} for telescopic semigroups as well as Theorem \ref{thm-Suzuki} for Suzuki semigroups provide us with estimates for the second Hamming weight of codes in the array of AG codes corresponding to these numerical semigroups (see~\cite{K-P}). We recall briefly the definition of the generalized (Hamming) weights. In fact the support of a linear code $C$ is defined as
\[
{\rm supp}(C):=\{i \mid c_{i}\neq 0\;\;\mbox{for some ${\bf c}\in C$}\},
\]
and the $r^{\mathrm{th}}$ generalized weight of $C$ is then 
\[
{\mathrm d}_{r}(C):=\min\{\#\,{\rm supp}(C')\mid \mbox{$C'\preceq C$ with ${\rm dim}(C')=r$}\}, 
\]
where $C'\preceq C$ denotes that $C'$ is a linear subcode of $C$. 
Observe that the above definition only makes sense if $r\leq k$, where $k$ is the dimension of $C$. 

Let $C_{a}$ be a code in an array of codes as in~\cite{K-P} with associated semigroup $\Gamma$. For example, $C_{a}$ may be a one-point AG code associated to a divisor of the form $G=aP$. In this case $\Gamma$ would be the Weierstrass semigroup of the underlying curve at $P$, as explained in the introduction. We will consider cases where $\Gamma$ is telescopic as in Corollary \ref{cor:E-tel} or free with assumptions as in Theorem \ref{ap-min}, so that the second Feng-Rao number equals $n_1$, the multiplicity of the semigroup. We will also consider the case of the Suzuki semigroups. 

Since free semigroups are symmetric (see~\cite[Chapter 8]{ns}), the bound for the second Hamming weight given by the second Feng-Rao number gives the exact value of the second Feng-Rao distance for half of the elements in the interval $[c,2c-1]$, where $c=2g$ and $a=2g-1+\rho$ with $\rho\in\Gamma\setminus\{0\}$ (see~\cite{F-M}). 

Let $a\geq c$. Corollary \ref{cor:E-tel} implies that 
\begin{equation}\label{bound-mult}
{\mathrm d}_{2}(C_{a}) \geq \delta_{FR}^{2}(a+1) \geq a+2-2g + E_2 = a+2-2g + n_1. 
\end{equation}
In contrast following \cite[Theorem 2.8]{K-P}, one deduces that 
\begin{equation}\label{bound-KP}
{\mathrm d}_{2}(C_{a}) \geq \delta_{FR}(a+2)\geq a+3-2g.
\end{equation}  
Since $n_1 > 1$ the bound obtained from Corollary \ref{cor:E-tel} is an improvement. 

Finally, the Griesmer order bound (introduced in~\cite{D}) for the case $r=2$ shows that 
\begin{equation}\label{bound-GOB}
{\mathrm d}_{2}(C_{a}) \geq \mathrm{GOB}(a+1) := \delta_{FR}(a+1) + \left\lceil\displaystyle\frac{\delta_{FR}(a+1)}{q}\right\rceil, 
\end{equation}
where the code is defined over the finite field $\mathbb{F}_{q}$\/. In the case of codes constructed from generalized Hermitian semigroups, 
the underlying field for the corresponding codes would be $\mathbb{F}_{q^r}$. Consequently we would divide by $q^{r}$ instead $q$ in the formula above. 

In the following examples we compare our bound (\ref{bound-mult}) for the second Hamming weight with (\ref{bound-KP}) and (\ref{bound-GOB}). 
We will denote $\delta_{FR}(a)=\delta_a$ and $\delta^{2}_{FR}(a)=\delta^{2}_{a}$ for simplicity. 

\begin{ex}\label{ex:GH}

Consider the generalized Hermitian semigroups as in section \ref{sec:GH}, that is 
\[
\Gamma = H_{q,r}=\langle q^{r-1}, q^{r-1}+q^{r-2}, q^r+1 \rangle
\]
where $q$ is the power of a prime number and $r \geq 3$ (note that if $r=2$ the semigroup is Hermitian and is generated by only two elements). \\

First, consider the case $q=2$ and $r=3$.  Here $\Gamma = \langle 4, 6, 9 \rangle$ with genus $g=6$. 
Table \ref{table:GH-23} shows the index $a$ of the code $C_{a}$ together with four bounds for the second Hamming weight, namely bounds 
(\ref{bound-KP}), (\ref{bound-GOB}), (\ref{bound-mult}) and finally the actual second Feng-Rao distance $\delta^{2}_{a+1}$\/. 
We observe that our bound (\ref{bound-mult}) is always better than Kirfel-Pellikaan bound (\ref{bound-KP}) and the Griesmer order bound (\ref{bound-GOB}). 
Additionally in this case, all of the values but one for our bound (\ref{bound-mult}) coincide exactly with the value of the second Feng-Rao distance. 

Next, consider the case $\Gamma = H_{2,4} = \langle 8, 12, 17 \rangle$ with genus $g=28$. 
Again for the majority of the values of $a$ our bound (\ref{bound-mult}) coincides exactly with the value of the second Feng-Rao distance. 
Additionally for all values of $a$ our bound improves upon the bounds (\ref{bound-KP}) and (\ref{bound-GOB}).  This is apparent for the values displayed in Table \ref{table:GH-24}.  Note that for the sake of brevity we omit displaying most of the columns in the table. 

We obtained similar results for higher values of $q$ and $r$ with the aid of \texttt{GAP}.

\begin{center}
\begin{table*}[!t]
\begin{tabular}{|c|cccccccccccc|}
\hline 
$a$ & 12 & 13 & 14 & 15 & 16 & 17 & 18 & 19 & 20 & 21 & 22 & 23  \\
\hline 
\hline 
$\delta_{a+2}$ & 4 & 4 & 6 & 6 & 8 & 8 & 9 & 10 & 12 & 12 & 13 & 13 \\
\hline 
$GOB(a+1):=\delta_{a+1}+\left\lceil\delta_{a+1}/q^r\right\rceil$ & 5 & 5 & 5 & 7 & 7 & 9 & 9 & 11 & 12 & 14 & 14 & 15 \\
\hline 
$a+1-2g+E_2$ & 6 & 7 & 8 & 9 & 10 & 11 & 12 & 13 & 14 & 15 & 16 & 17 \\
\hline 
$\delta^{2}_{a+1}$ & 6 & 8 & 8 & 9 & 10 & 11 & 12 & 13 & 14 & 15 & 16 & 17  \\
\hline 
\end{tabular}
\caption{Results for the case $H_{2,3}$.}\label{table:GH-23}
\end{table*}
\end{center}

\begin{center}
\begin{table*}[!t]
\begin{tabular}{|c|cccccccccccccccc|}
\hline 
$a$ & 56 & 57 & 58 & 59 & 60 & 61 & 62 & $\cdots$ & 104 & 105 & 106 & 107 & 108 & 109 & 110 & 111 \\
\hline 
\hline 
$\delta_{a+2}$ & 8 & 8 & 8 & 8 & 8 & 8 & 12 & $\cdots$ & 51 & 52 & 53 & 54 & 56 & 56 & 57 & 57 \\
\hline 
$GOB(a+1)$ & 9 & 9 & 9 & 9 & 9 & 9 & 9 & $\cdots$ & 54 & 55 & 56 & 57 & 58 & 60 & 60 & 61 \\
\hline 
$a+1-2g+E_2$ & 10 & 11 & 12 & 13 & 14 & 15 & 16 & $\cdots$ & 58 & 59 & 60 & 61 & 62 & 63 & 64 & 65 \\
\hline 
$\delta^{2}_{a+1}$ & 12 & 12 & 12 & 16 & 16 & 16 & 16 & $\cdots$ & 58 & 59 & 60 & 61 & 62 & 63 & 64 & 65 \\
\hline 
\end{tabular}
\caption{Results for the case $H_{2,4}$.}\label{table:GH-24}
\end{table*}
\end{center}
\end{ex}

In the next example we use Theorem \ref{thm-Suzuki} to obtain the bound 
\begin{equation}\label{bound-Suzuki}
{\mathrm d}_{2}(C_{a}) \geq \delta_{FR}^{2}(a+1) \geq a+2-2g + E_2 = a+2-2g + (p^{2n+1}-p^{2n}+p^{n}).
\end{equation}
 These codes are defined solely in the case where characteristic $p=2$; see~\cite{suzuki, gen-suzuki} for details.

\begin{ex}\label{ex:Suzuki}

Consider now the Suzuki semigroup from Section \ref{sec:Suzuki} for the case $p=2$ 
\[
\Gamma= S_{2,n} = \langle 2^{2n+1},2^{2n+1}+2^{n},2^{2n+1}+2^{n+1},2^{2n+1}+2^{n+1}+1\rangle. 
\]
 As noted earlier this semigroup is not telescopic. According to Theorem \ref{thm-Suzuki}, the second Feng-Rao number is $\mathrm E(\Gamma, 2)=2^{2n+1}-2^{2n}+2^{n}$. 

Consider the classical Suzuki curve with $n=1$, with Weierstrass semigroup $\Gamma=\langle 8,10,12,13\rangle$. 
This semigroup is free with genus $g=14$, conductor $c=28$, and second Feng-Rao number $E_{2}=6$. 
As is apparent in Table \ref{ex:Suzuki1}, our bound is better than both (\ref{bound-KP}) and (\ref{bound-GOB}).  Additionally we see that our bound equals the 
second Feng-Rao distance for a majority of the values. Note that in computing the Griesmer order bound we divide by the size of the finite field i.e. $2^{2n+1}$. 

We also performed computations for higher $n$, obtaining similar results with large tables. 
\begin{center}
\begin{table*}[!t]
\begin{tabular}{|c|cccccccccccccc|}
\hline 
$a$ &  28 & 29 & 30 & 31 & 32 & 33 & 34 & 35 & 36 & 37 & 38 & 39 & 40 & 41 \\
\hline 
\hline 
$\delta_{a+2}$ & 8 & 8 & 8 & 8 & 8 & 8 & 10 & 10 & 12 & 12 & 13 & 16 & 16 & 16 \\
\hline 
$GOB(a+1)$ & 7 & 9 & 9 & 9 & 9 & 9 & 9 & 12 & 12 & 14 & 14 & 15 & 18 & 18 \\
\hline 
$a+1-2g+E_2$ & 8 & 9 & 10 & 11 & 12 & 13 & 14 & 15 & 16 & 17 & 18 & 19 & 20 & 21 \\
\hline 
$\delta^{2}_{a+1}$ & 10 & 11 & 12 & 12 & 12 & 14 & 14 & 16 & 16 & 17 & 18 & 19 & 20 & 22 \\
\hline 
\hline 
\hline
$a$ & 42 & 43 & 44 & 45 & 46 & 47 & 48 & 49 & 50 & 51 & 52 & 53 & 54 & 55 \\
\hline 
\hline 
$\delta_{a+2}$ & 18 & 18 & 20 & 20 & 21 & 22 & 23 & 24 & 25 & 26 & 28 & 28 & 29 & 29 \\
\hline 
$GOB(a+1)$ & 18 & 21 & 21 & 23 & 23 & 24 & 25 & 26 & 27 & 29 & 30 & 32 & 32 & 33 \\
\hline 
$a+1-2g+E_2$ & 22 & 23 & 24 & 25 & 26 & 27 & 28 & 29 & 30 & 31 & 32 & 33 & 34 & 35 \\
\hline 
$\delta^{2}_{a+1}$ & 22 & 24 & 24 & 25 & 26 & 27 & 28 & 29 & 30 & 31 & 32 & 33 & 34 & 35 \\
\hline 
\end{tabular}
\caption{Results for the Suzuki semigroup $S_{2,1}$.}\label{ex:Suzuki1}
\end{table*}
\end{center}
\end{ex}

In conclusion, the bound for the second Hamming weight based on the second Feng-Rao number is better than those given by Kirfel-Pellikaan in~\cite{K-P} and 
the Griesmer order bound introduced in~\cite{D}, for AG codes coming from both generalized Hermitian curves and Suzuki curves. Additionally, recall that this bound 
equals the one given by the actual second Feng-Rao distance in most cases and specifically for all $a\geq c$.

\end{document}